\documentclass[11pt]{article}%
\usepackage{amsfonts}
\usepackage{amsmath}
\usepackage{amssymb}
\usepackage{graphicx}%
\newtheorem{theorem}{Theorem}
\newtheorem{acknowledgement}[theorem]{Acknowledgement}

\newtheorem{corollary}[theorem]{Corollary}

\newtheorem{proposition}[theorem]{Proposition}

\newenvironment{proof}[1][Proof]{\noindent\textbf{#1.} }{\ \rule{0.5em}{0.5em}}
\begin{document}

\title{The Pseudo-Character of the Weil Representation and its Relation with the
Conley--Zehnder Index}
\author{Maurice de Gosson\thanks{maurice.degosson@gmail.com}
\and Franz Luef\thanks{franz.luef@univie.ac.at}}
\maketitle

\begin{abstract}
We calculate the character of the Weil representation using previous results
which express the Weyl symbol of metaplectic operators in terms of the
symplectic Cayley transform and the Conley--Zehnder index.

\end{abstract}

\section{Introduction}

Let $\operatorname*{Sp}(2n,\mathbb{R})$ be the standard symplectic group: it
consists of all linear automorphisms of $\mathbb{R}^{2n}=T^{\ast}%
\mathbb{R}^{n}$ preserving the standard symplectic form $\sigma=\sum_{j=1}%
^{n}dp_{j}\wedge dx_{j}$ (the generic element of $\mathbb{R}^{2n}$ is
$z=(x,p)$). It is well-known that $\operatorname*{Sp}(2n,\mathbb{R})$ is a
connected Lie group and $\pi_{1}[\operatorname*{Sp}(2n,\mathbb{R})]$ is
isomorphic to the integer group $(\mathbb{Z},+)$ hence $\operatorname*{Sp}%
(2n,\mathbb{R})$ has covering groups $\operatorname*{Sp}_{q}(2n,\mathbb{R})$
of all orders $q=2,3,...,\infty$. It turns out that the double cover
$\operatorname*{Sp}_{2}(2n,\mathbb{R})$ can be faithfully represented by a
group of unitary operators on $L^{2}(\mathbb{R}^{n})$. This group is denoted
by $\operatorname*{Mp}(2n,\mathbb{R})$ and is called the Weil (or metaplectic)
representation of $\operatorname*{Sp}(2n,\mathbb{R})$. Its elements are called
metaplectic operators. The covering projection is denoted by $\Pi
:\operatorname*{Mp}(2n,\mathbb{R})\longrightarrow\operatorname*{Sp}%
(2n,\mathbb{R})$.

Trace formulas for diverse Weil representations have been recently obtained
(see for instance \cite{gurhad,tho}, also see \cite{luma}); such formulas are
important in many contexts, for instance in the theory of theta functions. In
this Note we study the analogue of trace formulas for the continuous case,
that is, for the full metaplectic representation. Of course, for metaplectic
operators the notion of trace does not make sense since such operators are not
of trace class. It is however possible to define what we call a
\textquotedblleft pseudo character\textquotedblright\ by the formula
\[
\operatorname*{Tr}(S)=\int_{\mathbb{R}^{n}}K_{S}(x,x)dx
\]
provided that $s=\Pi(S)$ has no eigenvalue equal to one; here $K_{S}$ is the
kernel of $S\in\operatorname*{Mp}(2n,\mathbb{R})$. We will see that the phase
of the \textquotedblleft pseudo-trace\textquotedblright\ in the right-hand
side is obtained in terms of the Conley--Zehnder index of symplectic paths,
familiar from the theory of periodic orbits of Hamiltonian systems (see
\cite{CZ,GGP,HWZ,Long} and the references in these works). This index has been
expressed in terms of the Leray--Maslov index (see \cite{Leray,JMPA1}) in de
Gosson \cite{JMP,Birk,RMP,JMPA2} and in de Gosson et al \cite{GGP}. The
Conley--Zehnder index also plays a key role in the semiclassical quantization
of chaotic Hamiltonian systems (the physicist's \textquotedblleft Gutzwiller
formula\textquotedblright) as has been recognized by Meinrenken
\cite{minibis,miniter}.

\section{Metaplectic operators as Weyl operators}

Recall that if $a\in\mathcal{S}^{\prime}(\mathbb{R}^{2n})$ the Weyl operator
with symbol $a$ is the operator $A:\mathcal{S}(\mathbb{R}^{n})\longrightarrow
\mathcal{S}^{\prime}(\mathbb{R}^{n})$ defined by%
\[
A=(2\pi)^{-n}\int_{\mathbb{R}^{2n}}a_{\sigma}(z)T(z)dz
\]
where $a_{\sigma}$ is the symplectic Fourier transform of $a$,%
\[
a_{\sigma}(z)=(2\pi)^{-n}\int_{\mathbb{R}^{2n}}e^{-i\sigma(z,z^{\prime}%
)}a(z^{\prime})dz^{\prime}%
\]
and $T(z)$ is the Heisenberg operator:%
\[
T(z)f(x^{\prime})=e^{-i(px^{\prime}-\frac{1}{2}px)}f(x^{\prime}-x).
\]
The distribution $a$ is the Weyl symbol of $A$.

In de Gosson \cite{Mp,Birk,JMP} metaplectic operators are studied from the
point of view of Weyl pseudo-differential calculus. The main results are
summarized in the following Theorem:

\begin{theorem}
\label{th1}Let $S\in\operatorname*{Mp}(2n,\mathbb{R})$ have projection
$\Pi(S)=s$ on $\operatorname*{Sp}(2n,\mathbb{R})$ such that $\det(s-I)\neq0$.
Then the symplectic Fourier transform of the Weyl symbol $a^{S}$ of $S$ is
given by the formula
\begin{equation}
a_{\sigma}^{S}(z)=\left(  \frac{1}{2\pi}\right)  ^{n}\frac{i^{\nu(S)}}%
{\sqrt{|\det(s-I)|}}e^{\frac{i}{2}Mz\cdot z} \label{1}%
\end{equation}
where
\begin{equation}
M=\tfrac{1}{2}J(s+I)(s-I)^{-1}\text{ \ , \ }J=%
\begin{pmatrix}
0 & I\\
-I & 0
\end{pmatrix}
\label{2}%
\end{equation}
The number $\nu(S)$, defined modulo $4$, is the Conley--Zehnder index of a
path joining the identity to $s$ in $\operatorname*{Sp}(2n,\mathbb{R})$ and
whose homotopy class depends on the choice of $S$ .
\end{theorem}

It is easily verified that $(s+I)(s-I)^{-1}\in\mathfrak{sp}(2n,\mathbb{R})$
(the Lie algebra of $\operatorname*{Sp}(2n,\mathbb{R})$), hence $M=M^{T}$ (the
mapping $s\longmapsto(s+I)(s-I)^{-1}$ is sometimes called the symplectic
Cayley transform). The index $\nu(S)$ corresponds to a choice of the argument
of $\det(s-I)$:%

\begin{equation}
\arg\det(s-I)\equiv(\nu(S)-n)\pi\text{ \ }\operatorname{mod}2\pi\text{.}
\label{Maslov2}%
\end{equation}

For a detailed study of this relationship see de Gosson \cite{JMPA2}, where
the Conley--Zehnder index is expressed in terms of the Leray--Maslov index
\cite{JMPA1} on the symplectic space $(\mathbb{R}^{2n}\oplus\mathbb{R}%
^{2n},\sigma\oplus(-\sigma))$.

Recall that the symbol $a$ of a Weyl operator $A$ is related to the kernel
$K_{A}$ of $A$ by the formula%
\begin{equation}
K_{A}(x,y)=\left(  \frac{1}{2\pi}\right)  ^{n}\int_{\mathbb{R}^{n}}%
e^{ip\cdot(x-y)}a(\tfrac{1}{2}(x+y),p)dp \label{kax}%
\end{equation}
(interpreted in the sense of distributions).

\section{Pseudo-Trace Formulas}

An immediate consequence of Theorem \ref{th1} is the following formula:

\begin{corollary}
Let $S\in\operatorname*{Mp}(2n,\mathbb{R})$ be as above. We have%
\begin{equation}
\operatorname*{Tr}(S)=\left(  \frac{1}{2\pi}\right)  ^{n}\frac{i^{\nu(S)}%
}{\sqrt{|\det(s-I)|}}. \label{3}%
\end{equation}

\end{corollary}

\begin{proof}
In view of formula (\ref{kax}) we can write
\begin{align*}
\operatorname*{Tr}(S)  &  =\int_{\mathbb{R}^{n}}K_{S}(x,x)dx\\
&  =\left(  \frac{1}{2\pi}\right)  ^{n}\int_{\mathbb{R}^{2n}}a(z)dz\\
&  =a_{\sigma}^{S}(0)
\end{align*}
hence (\ref{3}) in view of (\ref{1}).
\end{proof}

Assume now that $s\ell_{P}\cap\ell_{P}=\{0\}$ where $\ell_{P}=\{0\}\times
\mathbb{R}^{n}$; in the canonical symplectic basis of $(\mathbb{R}^{2n}%
,\sigma)$ we may identify $s$ with a block matrix $%
\begin{pmatrix}
A & B\\
C & D
\end{pmatrix}
$ with $\det B\neq0$, and $S\in\operatorname*{Mp}(2n,\mathbb{R})$ has
projection $\Pi(S)=s$ if and only if
\[
Sf(x)=\left(  \frac{1}{2\pi i}\right)  ^{n}i^{m}\sqrt{|\det B^{-1}|}%
\int_{\mathbb{R}^{n}}e^{iW(x,x^{\prime})}f(x^{\prime})dx^{\prime}%
\]
for $f\in\mathcal{S}(\mathbb{R}^{n})$ (Leray \cite{Leray}, de Gosson
\cite{Birk}); here
\[
W(x,x^{\prime})=\tfrac{1}{2}DB^{-1}x\cdot x-B^{-1}x\cdot x^{\prime}+\tfrac
{1}{2}B^{-1}Ax^{\prime}\cdot x^{\prime}%
\]
is the generating function of $s$ and $m$ is the Maslov index:
\[
\arg\det B^{-1}=m\pi\text{ \ }\operatorname{mod}2\pi.
\]
We will write from now on $s=s_{W}$ and $S=S_{W,m}$. It is proven in de Gosson
... that if $\det(s_{W}-I)\neq0$ then%
\begin{equation}
\nu(S_{W,m})=m-\operatorname*{Inert}W_{xx}^{\prime\prime} \label{Morse}%
\end{equation}
where $\operatorname*{Inert}W_{xx}^{\prime\prime}$ (the \textquotedblleft
Morse index\textquotedblright) is the signature of the Hessian matrix of the
mapping $x\longmapsto W(x,x)$. Thus:

\begin{corollary}
When $s=s_{W}$ and $\det(s_{W}-I)\neq0$ then%
\[
\operatorname*{Tr}(S_{W,m})=\frac{i^{m-\operatorname*{Inert}W_{xx}%
^{\prime\prime}}}{\sqrt{|\det(s_{W}-I)|}}.
\]

\end{corollary}

Note that we have, explicitly,%
\[
\det(s_{W}-I)=(-1)^{n}\det B\det(B^{-1}A+DB^{-1}-B^{-1}-(B^{T})^{-1})
\]
(see de Gosson \cite{Mp}, Lemma 4).

It turns out that we have the following factorization result (de Gosson
\cite{Mp}):

\begin{proposition}
Every $S\in\operatorname*{Mp}(2n,\mathbb{R})$ can be written (in infinitely
many ways) as a product $S=S_{W,m}S_{W^{\prime},m^{\prime}}$ such that
$\det(s_{W}-I)\neq0$ and $\det(s_{W^{\prime}}-I)\neq0$.
\end{proposition}

Using formula (\ref{Morse}) together with the product formula
\[
\nu(SS^{\prime})=\nu(S)+\nu(S^{\prime})+\tfrac{1}{2}\operatorname*{sign}%
(M+M^{\prime})
\]
proven in \cite{GGP,JMPA2} the result above allows the calculation of the
Conley--Zehnder index in the general case. The constructions in \cite{GGP} are
certainly useful in this context.

\begin{acknowledgement}
The first author has been financed by the Austrian Science Foundation FWF
(Projektnummer P20442-N13). The second author has been financed by the Marie
Curie Outgoing Fellowship PIOF 220464.
\end{acknowledgement}


\begin{thebibliography}{99}                                                                                               %
\bibitem {CZ}Conley, C. E., Zehnder, E.: Morse-type index theory for flows and
periodic solutions of Hamiltonian equations. Comm. Pure and Appl. Math.\ 37,
207--253 (1978)

\bibitem {JMPA1}de Gosson, M.: The structure of $q$-symplectic geometry. J.
Math. Pures et Appl. 71, 429--453 (1992)

\bibitem {Mp}de Gosson, M.: The Weyl Representation of Metaplectic operators.
Letters in Mathematical Physics 72 129--142 (2005)

\bibitem {Birk}de Gosson, M.: Symplectic Geometry and Quantum Mechanics.
Birkh\"{a}user, Basel, series \textquotedblleft Operator Theory: Advances and
Applications\textquotedblright\ (subseries: \textquotedblleft Advances in
Partial Differential Equations\textquotedblright), Vol. 166 (2006)

\bibitem {JMP}de Gosson, M., de Gosson, S.: An extension of the
Conley--Zehnder Index, a product formula and an application to the Weyl
representation of metaplectic operators. J. Math. Phys.,\ 47(12) (2006)

\bibitem {GGP}de Gosson, M., de Gosson, S., Piccione, P.: On a product formula
for the Conley--Zehnder Index of symplectic paths and its applications. To
appear in Ann. Global Analysis and Geom. 34, 167--183 (2008). Preprint 2006
(arXiv math.SG/0607024)

\bibitem {RMP}de Gosson, M.: Metaplectic Representation, Conley--Zehnder
Index, and Weyl Calculus on Phase Space. Rev. Math. Physics, 19(8), 1149--1188 (2007)

\bibitem {JMPA2}de Gosson, M.: On the usefulness of an index due to Leray for
studying the intersections of Lagrangian and symplectic paths. Journal de
Math\'{e}matiques Pures et Appliqu\'{e}s 91 (2009) 598--613 [Preprint
MPIM2007-119, Max Planck Institute for Mathematics preprint server:
http://www.mpim-bonn.mpg.de/ preprints/retrieve (2008)]

\bibitem {gurhad}Gurevich, S. and Hadani, R.: The geometric Weil
representation. Selecta Math.13(3) (2007) 465--481

\bibitem {HWZ}Hofer, H., Wysocki, K., Zehnder, E.: Properties of
pseudoholomorphic curves in symplectizations II: Embedding controls and
algebraic invariants. Geometric and Functional Analysis 2(5), 270--328 (1995)

\bibitem {Leray}Leray, J.: Lagrangian Analysis and Quantum Mechanics,\ a
mathematical structure related to asymptotic expansions and the Maslov index.
The MIT Press, Cambridge, Mass. (1981)

\bibitem {luma}Luef, F., Manin, Yu.: Quantum Theta Functions and Gabor Frames
for Modulation Spaces. Letters in Mathematical Physics 88(1--3) 132--161 (2005)

\bibitem {Long}Long, Y.: Index iteration theory for symplectic paths and
multiple periodic solution orbits. Frontiers of Math. 8, 341--353 (2006)

\bibitem {minibis}Meinrenken, E.: Trace formulas and the Conley--Zehnder
index. J. Geom. Phys. 13, 1--15 (1994)

\bibitem {miniter}Meinrenken, E.: Semiclassical principal symbols and
Gutzwiller's trace formula. Reports in Mathematical Physics 31, 279--295 (1992)

\bibitem {tho}Thomas, T. The character of the Weil representation. J. London
Math. Soc. (2) 77 (2008) 221--239.
\end{thebibliography}
\end{document}